\documentclass[12pt,reqno]{amsart}
\usepackage[a4paper]{geometry}
\usepackage{amsthm,hyperref,lmodern,mathrsfs,amssymb,amsmath,amsfonts}
\usepackage{enumerate}
\usepackage[latin1]{inputenc}
\usepackage{versions}

\usepackage{color}

\allowdisplaybreaks

\numberwithin{equation}{section}

\newcommand{\Hmm}[1]{\leavevmode{\marginpar{\tiny%
$\hbox to 0mm{\hspace*{-0.5mm}$\leftarrow$\hss}%
\vcenter{\vrule depth 0.1mm height 0.1mm width \the\marginparwidth}%
\hbox to 0mm{\hss$\rightarrow$\hspace*{-0.5mm}}$\\\relax\raggedright #1}}}

\newtheorem{theo}{Theorem}[section]
\newtheorem{lemme}[theo]{Lemma}

\def\N{{\mathbb{N}\ \!\!}}
\def\R{{\mathbb{R}}}

\def\D{{\mathcal{D}\ \!\!}}

\def\C{{\mathcal{C}\ \!\!}}

\def\FF{{\mathcal{F}\ \!\!}}

\def\Var{{\mathrm{{\rm Var}}}}

\def\and{{\mathrm{{\rm and}}}}

\begin{document}

\title[Eigenvalues estimates for one-dimensional diffusions]{A note on eigenvalues estimates for one-dimensional diffusion operators}

\author{Michel~Bonnefont} \address[M.~Bonnefont]{UMR CNRS 5251, Institut de Math\'ematiques de Bordeaux, Universit\'e Bordeaux 1, France}
\thanks{MB is partially supported by the French ANR-18-CE40-0012 RAGE project}
\email{\url{mailto:michel.bonnefont(at)math.u-bordeaux.fr}} \urladdr{\url{http://www.math.u-bordeaux.fr/~mibonnef/}}

\author{Ald\'eric~Joulin} \address[A.~Joulin]{UMR CNRS 5219, Institut de Math\'ematiques de Toulouse, Universit\'e de Toulouse, France}
\thanks{AJ is partially supported by the French ANR-18-CE40-006 MESA project}
\email{\url{mailto:ajoulin(at)insa-toulouse.fr}} \urladdr{\url{http://perso.math.univ-toulouse.fr/joulin/}}



\maketitle

\begin{abstract} Dealing with one-dimensional diffusion operators, we obtain upper and lower variational formulae on the eigenvalues given by the max-min principle, generalizing the celebrated result of Chen and Wang on the spectral gap. Our inequalities reveal to be sharp at least when the eigenvalues considered belong to the discrete spectrum of the operator, since in this case both lower and upper bounds coincide and involve the associated eigenfunctions. Based on the intertwinings between diffusion operators and some convenient gradients with weights, our approach also allows to estimate the gap between the two first positive eigenvalues when the spectral gap belongs to the discrete spectrum.
\end{abstract}

\section{Introduction}
\label{sect:intro}
Given a probability measure $\mu$ on $\R^n$ whose density with respect to the Lebesgue measure is (proportional to) $e^{-V}$, where $V$ is some smooth potential, one can associate a canonical self-adjoint diffusion operator,
\[
Lf = \Delta f- \nabla V \, \nabla f ,
\]
where $\Delta$ is the Laplacian and $\nabla$ stands for the Euclidean gradient on $\R^n$. Under some reasonable assumptions on $V$, it is well-known that the underlying Markov process converges in distribution to the invariant and reversible probability measure $\mu$. Moreover the speed of convergence in $L^2 (\mu)$ is given by the so-called spectral gap $\lambda_1(-L)$ of the operator, that is, its first positive eigenvalue. We refer to Section \ref{sect:prelim} for precise definitions together with more details about this large body of work. \smallskip

Except in some very particular situations, the exact value of the spectral gap is unknown in general and it is thus of particular interest to obtain some estimates on this quantity. In the one-dimensional case, things are a little simpler and it reveals to be of crucial importance since the study of the convergence to equilibrium of many high-dimensional Markovian models might be reduced to a careful analysis of an associated one-dimensional diffusion. Two decades ago, a lower variational formula for $\lambda_1(-L)$ was proposed by Chen and Wang \cite{chen_wang} through a coupling approach. It can be written as follows:
\[
\lambda_1(-L) \geq \sup_ b \inf _{x\in \R} V''(x) - \frac{L b(x)}{b(x)},
\]
where the supremum runs over all smooth positive function $b$. Moreover, if the spectral gap is attained, i.e., it belongs to the discrete spectrum of the operator $-L$, they proved that the equality holds by choosing $b$ as the derivative of an associated eigenfunction. Recently, such a formula has been revisited in \cite{bj,bjm2} by a totally different method, namely the use of the so-called intertwinings between gradients and operators. We mention that Chen-Wang's formula is not a special feature of the one-dimensional case since it holds in higher dimension, as it has been established in \cite{ABJ}. However, the analysis is rather delicate and involves some tedious technicalities since the objects of interest act on differential forms rather than on functions, cf. also \cite{bj2}, an instance which does not really occur in the one-dimensional case. \smallskip

Dealing with one-dimensional diffusion operators, the purpose of this note is to extend Chen-Wang's variational formula to higher eigenvalues given by the max-min principle. We provide lower and upper variational formulae on these eigenvalues and study the case of optimality. Let us briefly describe the main ideas of the proof of the two-sided estimates which is divided into several steps. Roughly speaking, the intertwining method consists in differentiating with some weight a diffusion operator and trying to write the result as a Schr\"odinger type operator acting on this derivative with weight. At this step, a key argument is to adapt to the weighted case a result established by Johnsen \cite{johnsen} who noticed that the intertwining is a unitary transformation between these two operators, so that their spectra coincide. According to this observation, the point then is to bound the associated multiplicative, or 0-order, potential appearing in this Schr\"odinger operator, yielding to a new diffusion operator. Then the preceding procedure has to be iterated for this new diffusion operator. At each step of the analysis, the weight appearing in the gradient used for the intertwining has to be chosen conveniently, since all the objects of interest strongly depend on it (together with the previous ones) in their very definition, in particular when studying the case of optimality in the estimates. \smallskip

The content of the paper is the following. In Section \ref{sect:prelim} we introduce the notation and the main ideas of the intertwining approach. Moreover we recall some basic facts on spectral analysis for self-adjoint operators, including the diffusion and Schr\"odinger operators emphasized by the intertwining. Section \ref{sect:main} is devoted to our main result, contained in Theorem \ref{theo:lambda-n}, in which some lower and upper variational formulae for eigenvalues given by the max-min principle are provided, together with a study of the possible case of optimality involving the underlying eigenfunctions. As mentioned above, the proof is quite technical and requires some preparation and intermediate results. In Section \ref{sect:ex-ap}, we illustrate our main result by some classical and less classical examples. On the one hand we consider the standard case of a uniformly convex potential, for which we recover the one-dimensional version of a result due to Milman \cite{milman}, comparing the eigenvalues of a diffusion operator with the ones of the Ornstein-Ulhenbeck operator associated to the standard Gaussian distribution. On the other hand, the intertwining approach is sufficiently robust to cover more interesting examples involving only convex potentials, e.g. the so-called Subbotin distribution, and even non-convex ones such as a double-well or an oscillating Gaussian potentials. Up to some universal constants, the examples under review exhibit sharp estimates with respect to the parameters of interest, in accordance with the celebrated Weyl law describing the asymptotic behavior of these eigenvalues. Finally, we propose in Section \ref{sect:ecart} a first attempt to study the gap between consecutive eigenvalues of a diffusion operator. The proof, which also uses the intertwining method, is based on a celebrated result of Brascamp and Lieb \cite{brascamp_lieb} about log-concavity properties of the ground state of a standard Schr\"odinger operator consisting of a Laplacian plus some potential. For the moment we are only able to estimate the gap between the two first positive eigenvalues, but the intertwining approach seems to be relevant to address this problem in full generality.

\section{Preliminaries and basic material}
\label{sect:prelim}

\subsection{Intertwinings}
Let $\C^\infty (\R)$ be the space of infinitely differentiable real-valued functions on the real line and let $\C_0 ^\infty (\R)$ be the subspace of $\C^\infty (\R)$ consisting of smooth compactly supported functions. The original diffusion operator we consider in this paper is defined on $\C^\infty (\R)$ by
\begin{eqnarray*}
L f & = & f'' - V' \, f' ,
\end{eqnarray*}
where $V \in \C^\infty (\R)$ is some smooth potential. The underlying invariant (and reversible) measure $\mu$ has Lebesgue density proportional to $e^{-V}$ and is assumed, throughout the paper, to be finite and renormalized into a probability measure. \smallskip

More generally if $\FF$ stands for the set of functions in $\C^\infty (\R)$ which do not vanish, then we introduce for a given function $a\in \FF$ the following Diffusion operator
\begin{equation}
\label{eq:La}
L_{a} f = f'' - V_{a} ' \, f' ,
\end{equation}
where
\begin{eqnarray*}
V_{a} & = & V + \log (a^2) .
\end{eqnarray*}
The potential $V_a$ is then in $\C^\infty (\R)$ and the invariant (and reversible) measure $\mu_{a}$, which is non-necessarily finite, has its Lebesgue-density proportional to $e^{-V_{a}}$. In particular choosing $a = 1$ we have $V_1 = V$ and thus $L_1 = L$ and $\mu_1 = \mu$. In other words when $a$ is chosen to be constant then the dynamics remains the same (up to some constant factors). \smallskip

Let us recall a result which is the cornerstone of the forthcoming analysis. It deals with the notion of intertwining between Diffusion operators and weighted gradients initiated in \cite{bj} in the one-dimensional case. Denote $\partial_a$ the usual gradient multiplied by a function $a\in \FF$, seen as a weight in the sequel, i.e. $\partial_a f = a f'$ (we will sometimes use the notation $\partial$ for the usual derivative). In the present language the result stands as follows (the proof is straightforward after some basic algebra).
\begin{lemme}\label{lemme:entrelacement}
Letting $a \in \FF$, if $b\in \FF$ is a given weight then we have the following intertwining relation between operators and weighted gradient: for every $f\in \C ^\infty (\R)$,
\label{lemme:entrelacement}
\begin{equation}
\label{eq:intert}
\partial_{b} L_{a} f = (L_{a b} - M_{a} ^{b}) \, \partial _{b} f ,
\end{equation}
where the operator $L_{ab}$ is obtained as in \eqref{eq:La}, the function $a$ being replaced by the product $ab$, and the smooth potential $M_{a} ^b$ is given by
\begin{eqnarray}
\label{eq:Ma_derivee}
\nonumber M_{a} ^b & =& V_{a} '' - b \, L_{a} (1/b) \\
\nonumber & = & V_{a} '' + \frac{L_{a b} (b)}{b} \\
& = & \frac{(-L_{a} h)'}{h'},
\end{eqnarray}
the last equality being obtained by rewriting $b$ as $b = 1/h'$ with $h' \in \FF$.
\end{lemme}
Above the operator $L_{ab} - M_{a} ^b$ is a Schr\"odinger type operator because of the presence of the multiplicative, or 0-order, potential $M_{a} ^b$. Actually, the intertwining \eqref{eq:intert} might be decomposed as the composition of two transformations (given naturally in this order but they actually commute):

$\circ$ first we consider the classical intertwining for the operator $L_a$, i.e., applying \eqref{eq:intert} without weight (that is, $b=1$) entails
$$
(L_a f)' = (L_a - M_a ^1) f'.
$$
Observing that the potential $M_a ^1$ rewrites as $V_a ''$, the latter identity is at the heart of the famous Bakry-\'Emery theory about the so-called $\Gamma_2$ calculus, cf. for instance \cite{BGL} for a nice introduction to the topic, with precise references. Let us mention that we cannot iterate in general the classical intertwining involving second derivatives, unless the potential $M_a ^1$ is constant (however we will see later in some convenient cases that we are able to iterate the more general intertwining \eqref{eq:intert} by choosing the weight $b$ so that the potential $M_a ^b$ is constant).

$\circ$ then the Schr\"odinger operator $L_{ab} - M_{a} ^b$ is the so-called (Doob) $h$-transform with $h=1/b$ of the operator $L_a - M_a ^1$ obtained at the previous step, that is,
$$
b \, (L_a - M_a ^1) \left( f/b \right) = (L_{ab} - M_{a} ^b ) f ,
$$
with the potential $M_a ^b$ rewritten as $M_a ^b = M_a ^1 - b L_a (1/b)$. Recall that the Doob transform exhibits an underlying group structure: if $h, k \in \FF$ are two given functions, then the $hk$-transform is nothing but the $h$-transform of the $k$-transform (it is also the $k$-transform of the $h$-transform). In particular the $h$-transform and the original operator are the same if and only if $h$ is constant. \smallskip

\noindent According to these observations, we understand why the operator $L_{ab}$ does not see the different roles of $a$ and $b$ and why the potentials $M_{a} ^b$, $M_1 ^{ab}$, $M_{ab} ^1$ and $M_b ^{a}$ all differ.

\subsection{Spectral analysis}
Let us recall some basic material on the spectrum of self-adjoint operators. By an essentially self-adjoint operator $T$, we mean that the operator $T$ initially defined on $\C_0 ^\infty (\R)$ admits a unique self-adjoint extension (still denoted $T$) with domain $\D(T) \subset L^2 (\nu)$, where $\nu$ is some non-negative measure on $\R$, possibly infinite, in which $\C ^\infty _0 (\R)$ is dense for the norm induced by $T$, i.e.,
$$
\Vert f \Vert _{D(T)} = \sqrt{\Vert f \Vert _{L^2 (\nu)} + \Vert Tf \Vert _{L^2 (\nu)}}.
$$
Above and in the sequel we use a general notation $T$ in order to unify the two cases we are interested in:
\begin{itemize}
\item[$(i)$] diffusion operators: $T = -L_{a}$ is associated to the measure $\nu = \mu_{a}$, assumed to be finite;
\item[$(ii)$] Schr\"odinger operators: $T = -L_{ab} + M_{a} ^b$ with potential $M_{a} ^b \in \C^\infty (\R)$, is endowed with the measure $\nu = \mu_{ab}$.
\end{itemize}
The operators of the type $(i)$ and $(ii)$ are symmetric and non-negative on $\C ^\infty _0 (\R)$. Indeed this is a classical observation for diffusion operators, together with the symmetry of the Schr\"odinger operators of $(ii)$. However the non-negativity of the latter is a bit more subtle since in the one-dimensional case it is a consequence of the intertwining of Lemma \ref{lemme:entrelacement}: for all $f, g\in \C ^\infty _0 (\R)$,
\begin{eqnarray*}
\int_\R f \, (-L_{ab} + M_a ^b) g \, d\mu_{ab} & = & \int_\R \partial_b h_1 \, (-L_{ab} + M_a ^b) \, \partial_b h_2 \, d\mu_{ab} \\
& = & \int_\R \partial_b h_1 \, \partial _b (-L_{a} h_2) \, d\mu_{ab} \\
& = & \int_\R h_1 ' \, (-L_{a} h_2)' \, d\mu_{a} \\
& = & \int_\R L_a h_1 \, L_{a} h_2 \, d\mu_{a} ,
\end{eqnarray*}
where $h_1 , h_2 \in \C ^\infty (\R)$ are some primitive functions of $f/b$ and $g/b$, respectively. Hence they are essentially-self-adjoint, cf. \cite{helffer_book}. \smallskip

The spectrum of the self-adjoint operator $T$, denoted $\sigma (T)$, corresponds to the complement of the set of $\lambda \in \R$ such that the operator $T-\lambda I$ is invertible from $\D (T)$ to $L^2(\mu)$ and has a continuous inverse. As $T$ is non-negative, $\sigma (T)$ is a subset of $[0,+\infty)$ and is divided into two disjoint parts: the discrete spectrum $\sigma_{disc} (T)$, that is, the set of isolated eigenvalues with finite multiplicity ($=1$ in our one-dimensional setting; by an eigenfunction associated to an eigenvalue $\lambda$ we mean some non identically null $g\in \D (T)$ such that $Tg = \lambda g$) and the essential spectrum $\sigma _{ess}(T)$, i.e., the complement of the discrete spectrum consisting of limit points in $\sigma (T)$ (the case of eigenvalues with infinite multiplicity cannot occur in the one-dimensional case, the dimension of the associated eigenspace being at most 2 according to the general theory of second order
linear differential equations). To identify whether or not a given real number belongs to the spectrum, the renown Weyl criterion states that $\lambda \in \sigma (T)$ if and only if there exists a sequence $(g_n)_{n\in \N} \subset \D (T)$ with $\Vert g_n \Vert _{L^2 (\nu)} = 1$ such that $$
\lim_{n\to +\infty } \Vert T g_n - \lambda g_n \Vert _{L^2 (\nu)} = 0.
$$
Moreover, $\lambda \in \sigma_{ess} (T)$ if and only if, in addition to the above properties, the sequence $(g_n)_{n\in \N}$ has no convergent subsequence. \smallskip

The bottom of the spectrum is given by the following variational formula
$$
\lambda_0 (T) = \inf _{f\in \D(T)} \frac{\int_{\R} f\, Tf \, d\nu}{\int_{\R } f^2 \, d\nu} .
$$
For instance in the diffusion case $(i)$ we have $\lambda_0 (-L_a) = 0$ and the constants are the associated eigenfunctions whereas in the Schr\"odinger setting $(ii)$ we only have $\lambda_0 (-L_{ab} + M_a ^b) \geq 0$. Coming back to the general case, the higher eigenvalues below the bottom of the essential spectrum are given by the well-known max-min principle, cf. \cite{helffer_book}. Denote $\perp _\nu$ the orthogonality induced by the scalar product in $L^2(\nu)$.
\begin{theo}[Max-min principle]
\label{theo:CF}
Consider the following variational formulae: for every $n \in \N^*$,
$$
\lambda_{n} (T) = \sup_{g_0 , g_1 , \ldots, g_{n-1} \in L^2 (\nu)} \, \underset{{\underset{f \perp_\nu  g_i, \, i = 0, \ldots, n-1}{f\in \D(T)}}}{\inf} \frac{\int_{\R} f\, Tf \, d\nu}{\int_{\R } f^2 \, d\nu} .
$$
Then either,

$(a)$ $\lambda_{n} (T)$ is the $(n+1)^{th}$ eigenvalue when ordering the eigenvalues in increasing order and $T$ has a discrete spectrum in $[0,\lambda_{n} (T)]$, i.e., $\sigma_{ess} (T) \cap [0,\lambda_{n} (T)] = \emptyset$;

\noindent or,

$(b)$ it is itself the bottom of the essential spectrum $\sigma_{ess}(T)$, all the $\lambda_m (T)$ coinciding with $\lambda_n (T)$ when $m \geq n$, and there are at most $n$ eigenvalues in $\sigma_{disc}(T) \cap [0,\lambda_{n} (T)]$.
\end{theo}
By abuse of language, all the elements $\lambda_n (T)$ arising in the max-min principle will be called eigenvalues in the remainder of the paper. \smallskip

As usual a density argument allows to take the infimum above over $\C_0 ^\infty (\R)$ instead of $\D(T)$. Note also that the supremum is realized when the $g_i$ are the associated eigenfunctions and it holds at least if the spectrum is discrete, i.e. $\sigma_{ess} (T) = \emptyset$ - in this case $(b)$ does not occur and the supremum (resp. infimum) is a maximum (resp. mimimum), justifying the standard terminology ``max-min principle". In particular the sequence of eigenvalues tends to infinity as $n$ goes to infinity. In the diffusion framework $(i)$, $\sigma_{ess} (T) = \emptyset$ if $V_a$ is uniformly convex or when
\begin{equation}
\label{eq:ess_empty}
\lim_{\vert x\vert \to +\infty} \frac{1}{2} \, V_a ' (x) ^2 - V_a ''(x) = +\infty,
\end{equation}
cf. \cite{helffer_book}. For instance a classical framework involving a discrete spectrum is the standard Gaussian case, i.e. the original potential $V$ is uniformly convex and given by $V = \vert \cdot \vert ^2 /2$ and $\mu$ is the standard Gaussian probability distribution. We have $\sigma(-L) = \N$ and the associated eigenfunctions are the Hermite polynomials. Another example of interest satisfying this time \eqref{eq:ess_empty} is the so-called Subbotin (or exponential-power) distribution with potential of the type $V = \vert \cdot \vert ^\alpha /\alpha$ with $\alpha >1$ (for $\alpha \in (1,2)$ its regularized version at the origin has to be considered), for which the eigenvalues are not known explicitly. We will come back to these basic examples later. \smallskip

The first positive eigenvalue is of crucial importance since it governs the exponential speed of convergence in $L^2(\nu)$ of the semigroup $(e^{- t T})_{t\geq 0}$. In the Schr\"odinger case $(ii)$, if $\lambda_0 (- L_{ab} + M_{a} ^b) >0$ then we have for every $f\in L^2 (\mu_{ab})$,
\begin{equation}
\label{eq:expo_cv_schro}
\Vert e^{t(L_{ab} - M_{a} ^b)} f \Vert _{L^2 (\mu_{ab})} \leq e^{-\lambda_0 (- L_{ab} + M_{a} ^b) t} \, \Vert f \Vert _{L^2 (\mu_{ab})} ,
\end{equation}
whereas in the diffusion setting $(i)$ the case $n=1$ is concerned and corresponds to the existence of the so-called spectral gap of the operator $-L_{a}$, i.e., $\lambda_1 (-L_{a}) >0$. Indeed for every function $f\in L^2 _0 (\mu_{a}) = \{ f \in L^2 (\mu_{a}) : f \perp _{\mu_a} 1 \}$, we have
\begin{equation}
\label{eq:expo_cv}
\Vert e^{ t L_{a} f } \Vert _{L^2 (\mu_{a})} \leq e^{-\lambda_1 (-L_{a}) t} \, \Vert f \Vert _{L^2 (\mu_{a})} .
\end{equation}
The existence of a spectral gap means that the first eigenvalue 0 is isolated: $0 \in \sigma_{disc} (-L_{a})$. It is the case as soon as the potential $V_{a}$ is uniformly convex, an instance of the celebrated Bakry-\'Emery criterion, or only convex \cite{kls,bobkov,BBCG} (such an assumption might be weakened to require only the convexity at infinity, at the price of a perturbation argument). Certainly, a stronger condition ensuring the existence of the spectral gap is $\lambda_1 (-L_{a}) \in \sigma_{disc} (-L_{a})$. More generally, as mentioned in the Introduction, a relevant criterion ensuring its existence is the formula of Chen and Wang, cf. \cite{chen_wang} and also \cite{bj, bjm2} where it has been revisited: if there exists a weight $b \in \FF$ such that $\inf_{x\in \R} \, M_{a} ^b (x) >0$, then
\begin{equation}
\label{eq:lambda1}
\lambda_1 (-L_{a}) \geq \inf_{x\in \R} \, M_{a} ^b (x) .
\end{equation}
Recently such an inequality has been successfully used to estimate the spectral gap in a large variety of examples, cf. \cite{bj, bjm2}. \smallskip

Another feature of the spectral gap in the diffusion case $(i)$ resides in the relationship with the regularity of the solution to the Poisson equation. More precisely if $\lambda_1 (-L_{a}) >0$ then for every centered function $f\in \C_0 ^\infty (\R)$ the Poisson equation
\begin{equation}
\label{eq:pois}
-L_{a} g = f,
\end{equation}
admits a unique solution $g  = (-L_{a})^{-1} f \in \D _0(-L_{a})$ which is in $\C^\infty (\R)$, where $\D _0(-L_{a}) = \{ f\in \D (-L_{a}) : f \perp_{\mu_a} 1 \}$. Here, the operator $(-L_{a})^{-1}$ or more generally $(-L_{a})^{-\alpha}$ for $\alpha>0$ is well-defined on $L^2 _0 (\mu_{a})$ as a Riesz-type potential:
\begin{equation}
\label{eq:Riesz}
(-L_{a})^{-\alpha} = \frac{1}{\Gamma(\alpha)} \, \int_0^{+\infty} t^{\alpha-1} \, e^{ t L_{a}} \, dt,
\end{equation}
where $\Gamma$ is the Gamma function $\Gamma (\alpha) = \int_0^{+\infty} t^{\alpha-1} e^{-t} \, dt$, cf. e.g. \cite{BGL}, since the inequality \eqref{eq:expo_cv} leads to the boundedness in $L^2_0 (\mu_{a})$ of the operator $(-L_{a})^{-\alpha}$, $\alpha >0$. \\ In the Schr\"odinger case $(ii)$ we use \eqref{eq:expo_cv_schro} and the same analysis as above remains valid provided the first eigenvalue $\lambda_0 (-L_{ab} + M_{a} ^b)$ is positive, at the price of removing the orthogonality condition $f \perp_{\mu_a} 1$ and replacing the spectral gap by $\lambda_0 (-L_{ab} + M_{a} ^b)$. Such Riesz-type representation will be used when necessary in the sequel.

\section{Main result}
\label{sect:main}
Before stating our main result Theorem \ref{theo:lambda-n} about eigenvalues estimates, let us present first some important ingredients that will be used in its proof.

\subsection{Johnsen's theorem revisited}
The first result that will be used in the sequel is Johnsen's theorem \cite{johnsen} which identifies the spectrum of the two operators involved in the (multidimensional version of the) intertwining of Lemma \ref{lemme:entrelacement}. Considering the original operator $L$ in the one-dimensional case, his result can be simplified (the restriction of the operators to the subspace generated by gradients is no longer required in dimension 1, since every smooth function is the gradient of its primitive functions; such an observation has already been used above when considering the essential self-adjointness property) and applies under a set of equivalent assumptions, among them $\lambda_1 (-L)>0$. His result is then the following.
\begin{theo}[Johnsen]
\label{theo:Johnsen_initial}
If a spectral gap holds, i.e., $\lambda_1 (-L)>0$, then we have equality of the spectra
$$
\sigma (-L) \backslash \{ 0 \} = \sigma ( - L + V'' ),
$$
with the analogous relation for the essential spectra. In terms of the eigenvalues given by the max-min principle, we have for every $n\in \N^*$,
$$
\lambda_n (-L ) = \lambda_{n-1} (- L + V'' ) .
$$
\end{theo}
Roughly speaking, the main idea of Johnsen is to observe that we have
$$
- L + V'' = U \, (-L)|_{\D _0(-L)} \, U^* ,
$$
where $U$ is the so-called Riesz transform
$$
U = \partial (-L)^{-1/2} = (-L + V'') ^{-1/2} \partial ,
$$
correctly defined on $\D _0 (-L)$ (since it is assumed that $\lambda_1 (-L) >0$) and with values in $\D (-L + V'')$, which is a unitary transformation (a surjective isometry). In other words the operators $(-L) |_{\D _0(-L)}$ and $- L + V'' $ are unitarily equivalent. Then for every $\lambda\in \R$ and every sequence $(f_n)_{n\in \N} \subset \D _0(-L)$ we have
$$
\Vert (- L + V'' - \lambda I) \, g_n \Vert _{L^2 (\mu)} = \Vert (- L - \lambda I) \, f_n \Vert _{L^2 (\mu)}, \quad \mbox{with} \quad g_n = U f_n ,
$$
and Weyl's criterion allows to conclude. \smallskip

In the present work, we need an analogue of Johnsen's theorem \ref{theo:Johnsen_initial} adapted to the presence of a weight in the gradient appearing in the intertwining of Lemma \ref{lemme:entrelacement}. The exact computation of the spectral gap $\lambda_1 (- L_{a})$ in the assumption below being of difficult access in general, we rather use for applications Chen-Wang's formula \eqref{eq:lambda1}, i.e. if there exists $b \in \FF$ such that $\inf_{x\in \R} \, M_{a} ^b (x) >0$, then
$$
\lambda_1 (-L_{a}) \geq \inf_{x\in \R} \, M_{a} ^b (x) .
$$
The result is stated as follows.
\begin{theo}[Johnsen's theorem revisited]
\label{theo:johnsen}
Let $a\in \FF$ be such that the measure $\mu_a$ is finite. Assume moreover that $\lambda_1 (- L_{a}) >0$. Then for all weight $b\in \FF$ we have
$$
\sigma (-L_{a}) \backslash \{ 0\} = \sigma (- L_{ab} + M_{a} ^b),
$$
with the analogous relation for the essential spectra. In terms of the eigenvalues given by the max-min principle, we have for every $n\in \N^*$,
$$
\lambda_n (-L_{a}) = \lambda_{n-1} (- L_{ab} + M_{a} ^b) .
$$
\end{theo}

\begin{proof}
To get the result we proceed similarly to Johnsen, where this time we have
$$
- L_{ab} + M_{a} ^b = U \, (-L_{a})|_{\D _0 (-L _{a})} \, U^{*},
$$
the operator $U$ being the weighted Riesz transform
$$
U = \partial _{b} (-L_{a})^{-1/2} = (-L_{ab} + M_{a} ^b)^{-1/2} \partial_{b},
$$
defined from $\D _0 (-L _{a})$ to $\D (-L_{ab} + M_{a} ^b)$, which is also a unitary mapping.
\end{proof}
The form of the previous unitary mapping is not innocent, since it is nothing but the composition of the two unitary transformations already describes above: the classical intertwining and the $h$-transform.

\subsection{A key lemma}
Another result of interest is the following key lemma, emphasizing a particular property of the underlying eigenfunction associated to the first positive eigenvalue. Although it is part of the folklore, cf. for instance \cite{chen_wang,miclo,bj,bir}, let us give a proof of this property for completeness since it will be used many times in the sequel.
\begin{lemme}
\label{lemme:g1}
We have the following characterizations:
\begin{itemize}
\item[$\circ$] In the diffusion case $(i)$:
\begin{itemize}
\item if $g_1 ^{a}$ is an eigenfunction associated to the spectral gap $\lambda_1 (- L_{a})$, then up to a change of sign we have $(g_1 ^{a})' >0$;
\item reciprocally if $g^{a}$ is an eigenfunction of the diffusion operator $- L_{a}$ such that $(g^{a}) ' >0$ (up to a change of sign) then the corresponding eigenvalue is the spectral gap $\lambda_1 (-L_{a})$.
\end{itemize}
\item[$\circ$] In the Schr\"odinger case $(ii)$:
\begin{itemize}
\item if $\tilde{g_0}^{a,b}$ is an eigenfunction associated to the eigenvalue $\lambda_0 (-L_{ab} + M_{a} ^b)$ (this eigenfunction is called the ground state), then up to a change of sign we have $\tilde{g_0}^{a,b} >0$;
\item reciprocally if $\tilde{g}^{a,b}$ is an eigenfunction of the Schr\"odinger operator $-L_{ab} + M_{a} ^b$ such that $\tilde{g}^{a,b} >0$ (up to a change of sign) then the corresponding eigenvalue is $\lambda_0 (-L_{ab} + M_{a} ^b)$, the bottom of the spectrum of the operator $-L_{ab} + M_{a} ^b$.
\end{itemize}
\end{itemize}
\end{lemme}
\begin{proof}
First note that according to the intertwining of Lemma \ref{lemme:entrelacement}, both problems for cases $(i)$ and $(ii)$ are actually equivalent, the relationship between $g_1 ^{a}$ and $\tilde{g_0} ^{a,b}$ being $\tilde{g_0} ^{a,b} = \partial_{b} g_1 ^{a}$. Indeed we have
$$
(- L_{ab} + M_{a} ^b) \, (\partial_b g_1 ^a) = - \partial_b L_a g_1 ^a = \lambda_1 (-L_a) \, \partial_b g_1 ^a = \lambda_0 (- L_{ab} + M_a ^b) \, \partial_b g_1 ^a ,
$$
where to get the last equality we used Theorem~\ref{theo:johnsen} with $n=1$. \smallskip

Hence let us prove the desired conclusion only in the diffusion case $(i)$. Our attention is concentrated on the first item. We claim that in the variational formula of the spectral gap,
\begin{eqnarray*}
\lambda_1 (-L_{a}) & = & \inf_{f\in \D _0 (-L_{a})} \frac{- \int_\R f \, L_{a} f \, d\mu_{a}}{\int_\R f^2 \, d\mu_{a}} \\
& = & \inf_{f\in \D (-L_{a})} \frac{\int_\R {f'} ^2 \, d\mu_{a}}{\Var _{\mu_{a}} (f) } ,
\end{eqnarray*}
where $\Var _{\mu_{a}} (f)$ denotes the variance of $f$ under the measure $\mu_{a}$, i.e.,
$$
\Var_{\mu_{a}} (f) = \int_\R f^2 \, d\mu_{a} - \left( \int_\R f \, d\mu_{a} \right) ^2 ,
$$
the infimum might be taken over monotonic functions. Indeed if $f\in \D (- L _{a})$ is a given function, then a primitive function $g$ of the function $\vert f'\vert $ which belongs to the space $\D (-L_a)$ satisfies
$$
\int_\R {g'} ^2 \, d\mu_{a} = \int_\R {f'}^2 \, d\mu_{a} ,
$$
and also
$$
\Var_{\mu_{a}} (f) \leq \Var_{\mu_{a}} (g),
$$
since the variance trivially rewrites as
$$
\Var_{\mu_{a}} (f) = \frac{1}{2} \, \int_\R \int_\R \left( \int_x ^y f'(t) \, dt \right)^2 \, d{\mu_{a}} (x) d{\mu_{a}} (y) .
$$
Therefore the first eigenfunction is monotonic. Now assume by contradiction that there exists some $x_0\in \R$ such that $ (g_1 ^{a})'(x_0)=0$. Hence we have $g_1 ^{a} (x_0)\neq 0$ (otherwise $g_1 ^{a}$ would be identically 0 by the Cauchy-Lipschitz theorem). Since
$$
- L _{a} g_1 ^{a} = \lambda_1 (-L_{a}) \, g_1 ^{a},
$$
we obtain $(g_1 ^{a})''(x_0) \neq 0$ and thus $g_1 ^{a}$ admits a local extrema in $x_0$, which contradicts its monotonicity. \smallskip

\noindent Let us now briefly establish the converse. Since we have
$$
- L _{a} g ^{a} = \lambda (-L_{a}) \, g ^{a},
$$
with $\lambda (-L_{a}) \neq 0$, the eigenfunction $g ^{a}$ being non-constant, it yields
$$
\lambda_1 (-L_{a}) \leq \lambda (-L_{a}),
$$
by the very definition of the spectral gap. Then the reverse inequality is obtained by choosing in Chen-Wang's formula \eqref{eq:lambda1} the function $b = 1/ (g ^{a}) '$. The proof is complete.
\end{proof}

\subsection{Main result}

Now we are in position to state the main result of the present paper, contained in Theorem~\ref{theo:lambda-n}, corresponding to lower and upper variational formulae for the eigenvalues of the original operator $- L$, generalizing to higher eigenvalues the famous Chen-Wang formula \eqref{eq:lambda1} on the spectral gap. In particular we exhibit the conditions under which the optimality holds. Let us briefly introduce some notation. If for some $i\in \N^*$ we have $\lambda_i (-L) \in \sigma_{disc} (-L)$ (resp. $\lambda_1 (-L_a) \in \sigma_{disc} (-L_a)$ for some $a\in \FF$), we denote $g_i$ (resp. $g_1 ^a$) an eigenfunction (recall that the associated eigenspaces are one-dimensional and moreover by ellipticity these eigenfunctions are in $\C^\infty (\R)$ since the potentials $V$ and $V_a$ are in $\C^\infty (\R)$). By Lemma \ref{lemme:g1} we know that the eigenfunction associated to the spectral gap is of constant sign and does not vanish. To let the forthcoming recursive argument start consistently, we introduce below the artificial weight $a_0 = 1$ so that $L_{a_0} = L$, $\mu_{a_0} = \mu$ and $g_1 ^{a_0} = g_1$. Finally, anticipating on the notation of Theorem \ref{theo:lambda-n} below, we recall that by Chen-Wang's formula \eqref{eq:lambda1} the existence of some $a_i \in \FF$ such that $\inf_{x\in \R} \, M_{a_0 \ldots a_{i-1}} ^{a_i} (x) >0$ is a sufficient condition ensuring that $\lambda_1 (-L_{a_0 \ldots a_{i-1}}) >0$, provided the measure $\mu_{a_0 \ldots a_{i-1}}$ is finite.
\begin{theo}
\label{theo:lambda-n}
\textbf{Two-sided estimates}. Given $n\in \N^*$, we have the following estimates:
\begin{equation}
\label{eq:lambda_n_init}
\sup_{a_1, \dots, a_n} \, \sum_{i= 1} ^{n } \inf_{x\in \R} M_{a_0 \ldots a_{i-1}} ^{a_i}(x) \leq \lambda_n (-L) \leq \inf_{a_1, \dots, a_n} \, \sum_{i= 1} ^{n } \sup_{x\in \R} M_{a_0 \ldots a_{i-1}} ^{ a_i} (x),
\end{equation}
where the supremum in the left-hand-side (resp. the infimum in the right-hand-side) runs over all functions $a_1,\ldots, a_{n} \in \FF$ such that for all $i\in \{ 1, \ldots, n \}$,

$\circ$ the measure $\mu_{a_0 \ldots a_{i-1}}$ is finite;

$\circ$ the spectral gap exists, i.e., $\lambda_1 (-L_{a_0 \ldots a_{i-1}}) >0$.

\noindent Above we adopt the convention that the left-hand-side (resp. right-hand-side) is infinite if at least one of the $M_{a_0 \ldots a_{i-1}} ^{a_i}$ is not bounded from below (resp. above). \smallskip

\noindent \textbf{Optimality}. Given $n\in \N^*$, if $\lambda_i (-L) \in \sigma_{disc}(-L)$ for every $i \in \{ 1, \ldots, n \}$ then $\lambda_1 (- L_{a_0 \ldots a_{i-1}}) \in \sigma_{disc}(-L_{a_0 \ldots a_{i-1}})$ for every $i \in \{ 1, \ldots, n \}$, the $a_i$ being given recursively by
\begin{equation}\label{eq:ai}
a_i = \frac{1}{(g_1 ^{a_0 \ldots a_{i-1}}) {'}} , \quad i.e., \quad \partial _{a_i} \ldots \partial _{a_1} g_i = 1.
\end{equation}
In this case the equalities hold in \eqref{eq:lambda_n_init} and we have
\begin{equation}
\label{eq:gap_consecutive}
\lambda_n (-L) = \sum_{i= 1} ^{n} \lambda_{1} (- L_{a_0 \ldots a_{i-1}}) .
\end{equation}
\end{theo}
\begin{proof}
First let us establish the inequality in the left-hand-side of \eqref{eq:lambda_n_init}, the other sense being somewhat similar by reversing the sign in the forthcoming inequalities. Pick some functions $a_1,\ldots, a_n\in \FF$ such that for all $i\in \{ 1, \ldots, n \}$ the measure $\mu_{a_0 \ldots a_{i-1}}$ is finite together with $\lambda_1 (-L_{a_0 \ldots a_{i-1}}) >0$. The strategy of the proof of the desired lower bound is to proceed recursively on $i\in \{ 1, \ldots, n \}$. The assumptions under consideration for $i=1$ are $\lambda_1 (-L_{a_0}) >0$ and the finiteness of the measure $\mu_{a_0}$ (the latter is an hypothesis assumed through the whole paper), so that by Theorem \ref{theo:johnsen} applied with $(a,b) = (a_0,a_1)$ we get
$$
\lambda_n (- L_{a_0}) = \lambda_{n-1} (- L_{a_0 a_1} + M_{a_0} ^{a_1}) \geq \lambda_{n-1} (- L_{a_0 a_1} ) + \inf_{x\in \R} \, M_{a_0} ^{a_1} (x),
$$
the inequality being a direct consequence of the max-min principle. Next we iterate the procedure: since $\mu_{a_0 a_1}$ is assumed to be finite and $\lambda_1 (-L_{a_0 a_1}) >0$, Theorem \ref{theo:johnsen} applied now with $(a,b) = (a_0 a_1 ,a_2)$ gives
$$
\lambda_{n-1} (- L _{a_0 a_1}) = \lambda_{n-2} (- L_{a_0 a_1 a_2} + M_{a_0 a_1} ^{a_2}) \geq \lambda_{n-2} (- L_{a_0 a_1 a_2} ) + \inf_{x\in \R} \, M_{a_0 a_1 } ^{a_2} (x),
$$
so that combining both estimates leads to the inequality
$$
\lambda_n (- L_{a_0}) \geq  \lambda_{n-2} (- L_{a_0 a_1 a_2}) + \inf_{x\in \R} \, M_{a_0} ^{a_1} (x) + \inf_{x\in \R} \, M_{a_0 a_1} ^{a_2} (x).
$$
The mechanism of a recursive procedure is now clear and entails the desired lower bound after taking the supremum over such functions $a_i$, the final step of the recursion being
\begin{eqnarray*}
\lambda_n (- L_{a_0}) & \geq & \lambda_0 (-L_{a_0 \ldots a_n}) + \sum_{i=1} ^n \inf_{x\in \R} \, M_{a_0 \ldots a_{i-1}} ^{a_i} (x) \\
& \geq & \sum_{i=1} ^n \inf_{x\in \R} \, M_{a_0 \ldots a_{i-1}} ^{a_i} (x).
\end{eqnarray*}
Note that $\lambda_0 (-L_{a_0 \ldots a_{n}})$ is non-negative but has no reason a priori to vanish since we do not assume the finiteness of the measure $\mu_{a_0 \ldots a_n}$. Finally taking the suprema over the admissible functions $a_1,\ldots, a_n$ yields the desired result. \smallskip

From now on, let us focus our attention on the possible optimality in the inequalities \eqref{eq:lambda_n_init}. According to the identity \eqref{eq:Ma_derivee}, if we define the functions $h_i$ such that $h_i ' = 1/a_i \in \FF$, then the inequalities \eqref{eq:lambda_n_init} might be rewritten as
\begin{equation}
\label{eq:lambda_n}
\sup \, \sum_{i= 1} ^{n } \inf_{x\in \R} \, \frac{(-L_{a_0 \ldots a_{i-1}} h_i)' (x)}{h_i '(x)} \leq \lambda_n (-L) \leq \inf \, \sum_{i= 1} ^{n } \sup_{x\in \R} \, \frac{(-L_{a_0 \ldots a_{i-1}} h_i)' (x)}{h_i ' (x)}.
\end{equation}
Therefore the desired optimality result holds if we show that $\lambda_1 (-L _{a_0 \ldots a_{i-1}}) \in \sigma_{disc}(-L _{a_0 \ldots a_{i-1}})$ for every $i \in \{ 1, \ldots, n\}$, so that the extrema in the inequalities \eqref{eq:lambda_n} are realized for the functions $h_i$ chosen as the eigenfunctions $g_1 ^{a_0 \ldots a_{i-1}}$ associated to the $\lambda_1 (-L _{a_0 \ldots a_{i-1}})$. These functions verify $h_i ' \in \FF$ by Lemma~\ref{lemme:g1}. \smallskip

Let us start by choosing the function $h_1$ as an eigenfunction $g_1 ^{a_0}$ associated to $\lambda_1 (-L_{a_0})$. We apply Theorem~\ref{theo:johnsen} with $(a,b) = (a_0,a_1)$ to get for all $k \in \{ 1, \ldots, n \}$,
$$
\lambda_k (-L_{a_0}) = \lambda_{k-1} (-L_{a_0 a_1} + M_{a_0} ^{a_1})  = \lambda_{k-1} (-L_{a_0 a_1}) + \lambda_1 (-L_{a_0}),
$$
because with this choice we have $M_{a_0} ^{a_1} = \lambda_1 (-L_{a_0})$ according to \eqref{eq:Ma_derivee}. Since we assumed $\lambda_k (-L) \in \sigma_{disc}(-L_{a_0})$ for every $k \in \{ 1, \ldots, n\}$, one deduces that $\lambda_{k-1} (-L_{a_0 a_1}) \in \sigma_{disc}(-L_{a_0 a_1})$ for every $k \in \{ 1, \ldots, n \}$. In particular it contains the eigenvalue $\lambda_{1} (-L_{a_0 a_1})$ which reveals to be the spectral gap of the operator $-L_{a_0 a_1}$ since the measure $\mu_{a_0 a_1}$ is finite, i.e.,
$$
\mu_{a_0 a_1} (\R) = \int_\R \left( (g_1 ^{a_0}) ' \right) ^2 \, d\mu_{a_0} < +\infty ,
$$
the function $g_1 ^{a_0}$ being an eigenfunction of the operator $-L_{a_0}$. \smallskip

Now we go one step beyond. The preceding argument allows us to choose for $h_2$ an eigenfunction $g_1 ^{a_0 a_1}$ associated to $\lambda_{1} (-L_{a_0 a_1})$. The measure $\mu_{a_0 a_1}$ being finite, we can apply Theorem~\ref{theo:johnsen} with $(a,b) = (a_0 a_1,a_2)$ to get for every $k \in \{ 1, \ldots, n-1 \}$,
$$
\lambda_{k} (-L_{a_0 a_1}) = \lambda_{k-1} (- L_{a_0 a_1 a_2} + M_{a_0 a_1} ^{a_2} ) = \lambda_{k-1} (- L_{a_0 a_1 a_2}) + \lambda_{1} (-L_{a_0 a_1}) ,
$$
where we have $M_{a_0 a_1} ^{a_2} = \lambda_{1} (-L_{a_0 a_1})$ by the identity \eqref{eq:Ma_derivee}. As above, the fact that $\lambda_{k} (- L_{a_0 a_1}) \in \sigma_{disc}(-L_{a_0 a_1})$ for every $k \in \{ 1, \ldots, n-1 \}$ implies that $\lambda_{k-1} (- L_{a_0 a_1 a_2}) \in \sigma_{disc}(-L_{a_0 a_1 a_2})$ for every $k \in \{ 1, \ldots, n-1 \}$, and so $\lambda_{1} (- L_{a_0 a_1 a_2}) $ is the spectral gap of the operator $- L_{a_0 a_1 a_2}$ since the measure $\mu_{a_0 a_1 a_2}$ is also finite:
$$
\mu_{a_0 a_1 a_2} (\R) = \int_\R \left( (g_1 ^{a_0 a_1}) {'} \right)  ^2 \, d\mu_{a_0 a_1} < +\infty ,
$$
the function $g_1 ^{a_0 a_1}$ being an eigenfunction of the operator $-L_{a_0 a_1}$. \smallskip

Following the previous analysis, we iterate the argument by choosing for each function $h_i$ an eigenfunction $g_1 ^{a_0 \ldots a_{i-1}}$ associated to the eigenvalue $\lambda_{1} (-L_{a_0 \ldots a_{i-1}})$, which lies in the set $\sigma_{disc} (-L_{a_0 \ldots a_{i-1}})$ according to the preceding step of the recursion, the measure $\mu_{a_0 \ldots a_{i-1}}$ being also finite. Hence we can apply Theorem~\ref{theo:johnsen} with $(a,b) = (a_0 \ldots a_{i-1},a_i)$ to get for every $k \in \{ 1, \ldots, n-(i-1) \}$,
$$
\lambda_{k} (-L_{a_0 \ldots a_{i-1}}) = \lambda_{k-1} (- L_{a_0\ldots a_{i}} + M_{a_0 \ldots a_{i-1}} ^{a_i}) = \lambda_{k-1} (- L_{a_0 \ldots a_{i}}) + \lambda_{1} (-L_{a_0 \ldots a_{i-1}}),
$$
thanks to the identity \eqref{eq:Ma_derivee} giving us $M_{a_0 \ldots a_{i-1}} ^{a_i} = \lambda_{1} (-L_{a_0 \ldots a_{i-1}})$. Since the previous step in the iterative procedure tells us that $\lambda_k (- L_{a_0 \ldots a_{i-1}}) \in \sigma_{disc}(- L_{a_0 \ldots a_{i-1}})$ for every $k \in \{1, \ldots, n-(i-1) \}$, we get that $\lambda_{k-1} (- L_{a_0 \ldots a_{i}}) \in \sigma_{disc}(- L_{a_0 \ldots a_{i}})$ for every $k \in \{ 1, \ldots, n-(i-1) \}$ and in particular the eigenvalue $\lambda_{1} (- L_{a_0 \ldots a_{i}})$ belongs to this set and is the spectral gap of the operator $- L_{a_0 \ldots a_{i}}$, the measure $\mu_{a_0 \ldots a_{i}}$ being finite. As such, we arrive at the end at
\begin{eqnarray*}
\lambda_n (-L_{a_0}) & = & \lambda_{n-1} (-L_{a_0 a_1}) + \lambda_1 (-L_{a_0}) \\
& = & \lambda_{n-2} (-L_{a_0 a_1 a_2}) + \lambda_{1} (-L_{a_0 a_1}) + \lambda_1 (-L_{a_0}) \\
& = & \ldots \\
& = & \lambda_{1} (-L_{a_0 \ldots a_{n-1}}) + \ldots + \lambda_{1} (-L_{a_0 a_1}) + \lambda_1 (-L_{a_0}) ,
\end{eqnarray*}
i.e., the identity \eqref{eq:gap_consecutive} holds. \smallskip

To conclude the proof of the optimality result, we only need to show that for all $i \in \{ 1, \ldots, n \}$ the preceding choice of the $a_i = 1/ (g_1 ^{a_0 \ldots a_{i-1}}) {'} $
coincides with the well-defined system
$$
\partial _{a_{i}} \ldots \partial _{a_1} g_i  =1,
$$
and this might be established through a recursive argument again. First this is true for $i=1$ by definition. Then for $i=2$, the eigenfunction $g_1 ^{a_0 a_1}$ associated to $\lambda_1 (-L_{a_0 a_1})$ might be deduced from the intertwining of Lemma \ref{lemme:entrelacement} with $(a,b) = (a_0,a_1)$:
\begin{eqnarray*}
L_{a_0 a_1} ( \partial_{a_1} g_2) & = & \partial_{a_1} L_{a_0} g_2 + M_{a_0} ^{a_1} \, \partial_{a_1} g_2 \\
& = & - \left( \lambda_2 (-L_{a_0}) - \lambda_1 (-L_{a_0}) \right) \, \partial_{a_1} g_2 \\
& = & - \lambda_1 (-L_{a_0 a_1}) \, \partial_{a_1} g_2,
\end{eqnarray*}
by the identity \eqref{eq:gap_consecutive} applied at rank $n=2$. Therefore by Lemma \ref{lemme:g1} we have $g_1 ^{a_0 a_1} = \partial_{a_1} g_2$ (up to some change of sign) and the derivative of $\partial_{a_1} g_2$ does not vanish, inducing indeed the well-defined identity $\partial_{a_2} \partial_{a_1} g_2 = 1$. \smallskip

Then the key point of the recursive argument is to notice that with this choice of the $a_i$, we obtain an intertwining of higher order by successive differentiations since the $M_{a_0 \ldots a_{j-1}} ^{a_j} $ are constant and equal $\lambda_1 (-L_{a_0 \ldots a_{j-1}})$, as previously observed: for all $i \in \{ 2, \ldots, n \}$,
\begin{eqnarray*}
\partial_{a_{i-1}} \ldots \partial_{a_1} L_{a_0} f & = & L_{a_0 \ldots a_{i-1}} (\partial_{a_{i-1}} \ldots \partial_{a_1} f) - \left( \sum_{j=1} ^{i-1} M_{a_0 \ldots a_{j-1}} ^{a_j} \right) \, \partial_{a_{i-1}} \ldots \partial_{a_1} f \\
& = & L_{a_0 \ldots a_{i-1}} (\partial_{a_{i-1}} \ldots \partial_{a_1} f) - \left( \sum_{j=1} ^{i-1} \lambda_1 (-L_{a_0 \ldots a_{j-1}}) \right) \, \partial_{a_{i-1}} \ldots \partial_{a_1} f \\
& = & L_{a_0 \ldots a_{i-1}} (\partial_{a_{i-1}} \ldots \partial_{a_1} f) - \lambda_{i-1}(-L_{a_0}) \, \partial_{a_{i-1}} \ldots \partial_{a_1} f,
\end{eqnarray*}
where in the last equality we used \eqref{eq:gap_consecutive} at rank $i-1$. In other words the intertwining of order $i-1$ is related to the eigenvalue $\lambda_{i-1}(-L_{a_0})$. Finally taking $f = g_i$ we get
\begin{eqnarray*}
L_{a_1, \ldots, a_{i-1}} (\partial_{a_{i-1}} \ldots \partial_{a_1} g_i) & = & \partial_{a_{i-1}} \ldots \partial_{a_1} L_{a_0} g_i + \lambda_{i-1}(-L_{a_0}) \, \partial_{a_{i-1}} \ldots \partial_{a_1} g_i \\
& = & - \left( \lambda_{i}(-L_{a_0}) - \lambda_{i-1}(-L_{a_0})\right) \, \partial_{a_{i-1}} \ldots \partial_{a_1} g_i \\
& = & - \lambda_1 (- L _{a_0 \ldots a_{i-1}}) \, \partial_{a_{i-1}} \ldots \partial_{a_1} g_i .
\end{eqnarray*}
Hence by Lemma \ref{lemme:g1} we deduce that, up to some change of sign, we have $g_1 ^{a_0 \ldots a_{i-1}} = \partial_{a_{i-1}} \ldots \partial_{a_1} g_i$ and thus $\partial_{a_{i}} \ldots \partial_{a_1} g_i =1$. The proof of Theorem \ref{theo:lambda-n} is now achieved.
\end{proof}

Let us comment Theorem~\ref{theo:lambda-n}. To the authors knowledge, the two-sided estimates \eqref{eq:lambda_n_init} together with the optimality result consisting of decomposing the nth positive eigenvalue as the sum of $n$ spectral gaps of different operators, are new. Behind the optimality result is hidden an interesting property of the eigenfunctions $g_i$ associated to the operator $-L$, known in the famous Sturm-Liouville theory as a Chebyshev system in the sense of Karlin and Studden \cite{KS}. In other words, the proof above gives some interesting information on the oscillations of the $g_i$ in the sense that, beyond the strict monotonicity of $g_1$ (an instance of Lemma \ref{lemme:g1}), the following functions are also strictly monotone:
$$
\frac{g_2 '}{g_1' }, \quad \left(\frac{g_3 '}{g_1'} \right)' / \left( \frac{g_2 '}{g_1 '} \right)'  , \quad \left( \frac{ \left( \frac{g_4 '}{g_1'} \right)'  }{ \left( \frac{g_2 '}{g_1'} \right)'}  \right) ' / \left( \frac{ \left( \frac{g_3 '}{g_1'} \right)' }{ \left( \frac{g_2 ' }{g_1' }\right)' }  \right) ' , \quad \mbox{etc.}
$$
Certainly, such a property is trivial in the standard Gaussian case introduced previously since it is obviously satisfied by the Hermite polynomials. However the proof of Theorem \ref{theo:lambda-n} reveals that it concerns actually a much more general framework through the one-dimensional diffusion operators we investigate in the present paper. \smallskip

As a final remark of Theorem~\ref{theo:lambda-n}, we mention that the constraint on the last eigenvalue $\lambda_n (-L)$ might be relaxed in the optimality result. Indeed, let us only assume that $\lambda_i (-L) \in \sigma_{disc} (-L)$, for $i \in \{ 1, \ldots, n-1 \}$, so that $\lambda_n (-L)$ is allowed to be the bottom of $\sigma_{ess} (-L)$. Then we arrive \textit{a priori} at the following result (once again we use the notation involving the artificial weight $a_0 = 1$):
\begin{equation}
\label{eq:ecart}
\sup_{a_n \in \FF} \, \inf_{x\in \R} \, M_{a_0 \ldots a_{n-1}} ^{a_n} (x)\leq \lambda_n (-L) - \lambda_{n-1} (-L) \leq \inf_{a_n \in \FF} \, \sup_{x\in \R} \, M_{a_0 \ldots a_{n-1}} ^{a_n} (x),
\end{equation}
where the $a_i$ are given recursively by the well-defined system
\begin{equation}
\label{eq:ai_system}
\partial_{a_i} \ldots \partial_{a_1} g_i = 1, \quad i\in \{ 1, \ldots, n-1 \}.
\end{equation}
Such a result, written in this form, will be exploited in Section~\ref{sect:ecart} to address the problem of estimating the gap between consecutive eigenvalues. In this situation, we claim that the equality holds (at least) in the left-hand-side of \eqref{eq:ecart}, which might be rewritten, since the identity \eqref{eq:gap_consecutive} still holds under the present assumptions, as
\begin{equation}
\label{eq:lambda_an}
\lambda_1 (-L_{a_0 \ldots a_{n-1}}) = \sup_{a_n \in \FF} \, \inf_{x\in \R} \, M_{a_0 \ldots a_{n-1}} ^{a_n} (x).
\end{equation}
To establish the identity \eqref{eq:lambda_an}, it is sufficient to show the inequality $\leq$, the other being a direct consequence of Chen-Wang's formula \eqref{eq:lambda1} applied with $a = a_0 \ldots a_{n-1}$. According to the choice of the $a_i$, $i \in \{ 1,\ldots, n-1 \}$, the measure $\mu_{a_0 \ldots a_{n-1}}$ is finite and by \eqref{eq:gap_consecutive} we know that $\lambda_1 (-L_{a_0 \ldots a_{n-1}}) >0$ since $\lambda_{n-1} (-L)$ is assumed to belong to the set $\sigma_{disc} (-L)$ and thus is isolated. Hence applying Theorem \ref{theo:johnsen} gives with $(a,b)= (a_0 \ldots a_{n-1},1)$,
\[
\lambda_1 (-L_{a_0 \ldots a_{n-1}}) = \lambda_0 ( -L_{a_0 \ldots a_{n-1}} + V_{a_0 \ldots a_{n-1}} ''),
\]
since $M_{a_0 \ldots a_{n-1}} ^{1} = V_{a_0 \ldots a_{n-1}} ''$. Now choosing $ \lambda < \lambda_0 ( -L_{a_0 \ldots a_{n-1}} + V_{a_0 \ldots a_{n-1}} '')$ so that the operator $-L_{a_0 \ldots a_{n-1}} + V_{a_0 \ldots a_{n-1}} '' -\lambda I$ is invertible from $\D (-L_{a_0 \ldots a_{n-1}} + V_{a_0 \ldots a_{n-1}} '')$ to $L^2(\mu_{a_0 \ldots a_{n-1}})$ and taking some non identically null $g\in \C^\infty _0 (\R)$ with $ g\geq 0$, there exists a unique smooth function $f \in \D (-L_{a_0 \ldots a_{n-1}} + V_{a_0 \ldots a_{n-1}} '' )$ which is in $\C^\infty (\R)$, such that
$$
( -L_{a_0 \ldots a_{n-1}} + V_{a_0 \ldots a_{n-1}} '' -\lambda I ) \, f = g.
$$
Using its Riesz-type representation somewhat similar to \eqref{eq:Riesz}, the operator $(-L_{a_0 \ldots a_{n-1}} + V_{a_0 \ldots a_{n-1}} '' -\lambda I)^{-1}$ is positivity improving, i.e., $f>0$. Now the latter expression rewrites as $(-L_{a_0 \ldots a_{n-1}} + V_{a_0 \ldots a_{n-1}} '') \, f = \lambda f + g$ and since $g \geq 0$ we get
\[
\inf_{x\in \R} \, V_{a_0 \ldots a_{n-1}} '' (x)- \frac{L_{a_0 \ldots a_{n-1}} f(x)}{f(x)} \geq \lambda,
\]
i.e., $\inf_{x\in \R} \, M_{a_0 \ldots a_{n-1}} ^{a_n} (x) \geq \lambda$ when letting $a_n = 1/f\in \FF$. Finally taking the supremum over $a_n \in \FF$ and letting $\lambda $ tends to $\lambda_0 ( -L_{a_0 \ldots a_{n-1}} + V_{a_0 \ldots a_{n-1}} '') = \lambda_1(-L_{a_0 \ldots a_{n-1}})$ yields the desired inequality. Note that it is an interesting question to investigate whether the supremum in the above theorem is attained for some function $a_n \in \FF$.

\section{Examples}
\label{sect:ex-ap}
In this section we investigate some concrete situations where Theorem~\ref{theo:lambda-n} might be illustrated.

\subsection{The uniformly convex case}
Let us briefly consider the case of a uniformly convex potential, $\inf_{x\in \R} \, V''(x) >0$, for which we know that the spectrum of the operator $-L$ is discrete (so that the spectral gap exists). Choosing in Theorem \ref{theo:lambda-n} all the $a_i =1$ entails that $M_{a_0 \ldots a_{i-1}} ^{a_i} = V''$ for all $i \geq 1$, so that we get the following result.
\begin{theo}
\label{theo:milman}
Assuming that $\inf_{x\in \R} \, V''(x) >0$, we have for every $n\in \N^*$,
$$
n \, \inf_{x\in \R} \, V''(x) \leq \lambda_n (-L) \leq n \, \sup_{x\in \R} \, V''(x) ,
$$
with the convention that the right-hand-side is infinite if $V''$ is not bounded from above.
\end{theo}
In other words, if we denote $L^\rho$ the operator with potential $V = \rho \vert \cdot \vert ^2 /2$, where $\rho>0$, we have in this Gaussian setting $\lambda_n (-L^\rho ) = n \rho $ and the previous result means that the eigenvalues might be compared as follows:
\begin{equation}
\label{eq:compar_gauss}
\lambda_n (-L^{\rho_1} ) \leq \lambda_n (-L) \leq \lambda_n (-L^{\rho_2}) ,
\end{equation}
where $\rho_1 = \inf_{x\in \R} \, V''(x)$ et $\rho_2 = \sup_{x\in \R} \, V''(x)$ are assumed to be positive and finite. Such an observation has been put forward in a recent paper of Milman \cite{milman} in which he proceeds to a systematic comparison of eigenvalues for the Riemannian version of the operator $L$ (called a weighted Laplacian in his article) under some uniform convexity assumptions on the potential. \smallskip

Actually, the uniformly convex case is special since by Johnsen's theorem \ref{theo:Johnsen_initial}, we get for all $n\in \N^*$,
$$
\lambda_n (-L) = \lambda_{n-1} (-L + V''),
$$
from which the following result about the gap between consecutive eigenvalues is straightforward (we will see in Section \ref{sect:ecart} that estimating the gap between consecutive eigenvalues is much more complicated when dealing with non uniformly convex potentials).
\begin{theo}
\label{theo:ecart_UC}
Assuming that $\inf_{x\in \R} \, V''(x) >0$, then we have for every $n\in \N^*$,
$$
\inf_{x\in \R} \, V''(x) \leq \lambda_n (-L) - \lambda_{n-1} (-L) \leq \sup_{x\in \R} \, V''(x) ,
$$
with the convention that the right-hand-side is infinite if $V''$ is not bounded from above.
\end{theo}
Hence in the uniformly convex case Theorem \ref{theo:milman} is nothing but a direct consequence of Theorem \ref{theo:ecart_UC}, meaning that the two-sided estimates of Theorem \ref{theo:lambda-n} are somewhat useless in this context (however the optimality result is still interesting). The next part proposes to address more interesting examples beyond the uniformly convex case, for which Theorem \ref{theo:lambda-n} is really required to estimate $\lambda_n (-L)$.

\subsection{Subbotin distributions}
Now let us investigate how Theorem \ref{theo:lambda-n} might be applied to a non uniformly convex potential, through the example of $V = \vert \cdot \vert ^\alpha / \alpha$ where $\alpha \geq 1$. This is the so-called Subbotin (or exponential power) distribution (for $\alpha = 2$ we are reduced to the standard Gaussian case). The potential $V$ is strictly convex on $\R$ but its second derivative vanishes at infinity for $\alpha \in (1,2)$ and at zero for $\alpha >2$ (although $V$ is not smooth at the origin for $\alpha \in (1,2)$, this point might be ignored at the price of an unessential regularizing procedure). According to \eqref{eq:ess_empty} we have $\sigma_{ess} (-L) = \emptyset$ for $\alpha >1$ and thus the spectrum is discrete. The case of the Laplace distribution, i.e. $\alpha=1$, is somewhat particular since the essential spectrum is not empty (actually we have $\sigma_{disc} (-L) = \{ 0\}$ and $\sigma_{ess} (-L) = [1/4, +\infty)$) and therefore is not concerned with the forthcoming analysis. Let us concentrate our attention on the range $\alpha \in (1,2)$ only, the framework $\alpha >2$ being tackled differently in the next part. Our aim is to apply Theorem \ref{theo:lambda-n} to this example, or more precisely to obtain at least a relevant lower bound on $\lambda_n (-L)$ for $n\geq 1$.  We set $a_1 = e^{- \varepsilon_1 V}$ and for all $i \geq 2$,
$$
a_i = \exp \left( - \varepsilon_i \, \prod_{k=1} ^{i-1} (1-2\varepsilon_k) \, V \right) ,
$$
where the $\varepsilon_i \in (0,1/2)$ are some constants to be chosen conveniently later. Then for all $i \geq 2$  we have $a_i = e^{-\varepsilon_i V_{a_1 \ldots a_{i-1}}}$ where
$$
V_{a_1 \ldots a_{i-1}} = \prod_{k=1} ^{i-1} (1-2\varepsilon_k) \, V .
$$
Moreover we have for all $x\in \R$,
\begin{eqnarray*}
M_{a_1 \ldots a_{i-1}} ^{a_i} (x) & = & (1-\varepsilon_i) \left( V_{a_1 \ldots a_{i-1}} '' (x) + \varepsilon_i \, (V_{a_1 \ldots a_{i-1}}' (x))^2 \right) \\
& = & (1-\varepsilon_i) \left( \prod_{k=1} ^{i-1} (1-2\varepsilon_k) \, V '' (x)+ \varepsilon_i \, \prod_{k=1} ^{i-1} (1-2\varepsilon_k)^2 \, (V' (x))^2  \right) \\
& = & C_i \, \vert x\vert ^{\alpha-2} + D_i \, \vert x\vert ^{2(\alpha-1)} ,
\end{eqnarray*}
with the constants $C_i$ and $D_i$ defined by
$$
C_i = (\alpha-1) (1-\varepsilon_i) \prod_{k=1} ^{i-1} (1-2\varepsilon_k) \quad \mbox{and} \quad D_i = \varepsilon_i (1-\varepsilon_i) \prod_{k=1} ^{i-1} (1-2\varepsilon_k)^2 .
$$
The minimum of $M_{a_1 \ldots a_{i-1}} ^{a_i}$ is attained at point
$$
x= \left( \frac{(2-\alpha)C_i}{2(\alpha-1)D_i} \right)^{1/\alpha} ,
$$
so that
\begin{eqnarray*}
\inf_{x\in \R} \, M_{a_1 \ldots a_{i-1}} ^{a_i} (x) & = & \left( \frac{2-\alpha}{2(\alpha-1)} \right)^{1-2/\alpha} \, \frac{\alpha}{2(\alpha-1)} \, C_i \, \left( \frac{C_i}{D_i}\right) ^{1 - 2/\alpha} \\
& = & \left( \frac{2-\alpha}{2} \right)^{1-2/\alpha} \, \frac{\alpha}{2} \, (1-\varepsilon_i) \, \varepsilon_i ^{(2-\alpha)/\alpha } \, \prod_{k=1} ^{i-1} (1-2\varepsilon_k)^{2/\alpha} .
\end{eqnarray*}
Hence by Theorem \ref{theo:lambda-n} we have the following estimate: for all $n\geq 1$,
$$
\lambda_n (-L) \geq \left( \frac{2-\alpha}{2} \right)^{1-2/\alpha} \, \frac{\alpha}{2} \, \sum_{i=1} ^n  (1-\varepsilon_i) \, \varepsilon_i ^{(2-\alpha)/\alpha} \, \prod_{k=1} ^{i-1} (1-2\varepsilon_k)^{2/\alpha}.
$$
Now we need to find some $\varepsilon_i$ such that the series diverges (recall that $\lambda_n (-L) $ tends to infinity as $n$ tends to infinity) and, to that aim, assuming for instance $\sum_k \varepsilon_k < +\infty$ enforces the infinite product to be positive and thus the constraint $\sum_i \varepsilon_i ^{(2-\alpha)/\alpha } = + \infty$. Finally the choice $\varepsilon_k = (k+1)^{-\beta} /2 \in (0,1/2)$ for some $\beta$ in the range $(1,\alpha /(2-\alpha))$ allows us to get the estimate
$$
\lambda_n (-L) \geq C_{\alpha,\beta} \, n^{1-\beta (2/\alpha -1)},
$$
with $C_{\alpha,\beta} >0$ some explicit constant only depending on $\alpha$ and $\beta$. In particular choosing $\beta$ in a neighborhood of 1 (choosing directly $\beta =1$ in the very definition of $\varepsilon_i$ above is allowed but not relevant since it leads to a lower bound on $\lambda_n (-L)$ converging to 0 as $n$ is large) entails a lower bound that is close to optimality for large $n$ given by the rate $n^{2 - 2/\alpha }$. Indeed such a result, which is available for all $\alpha >1$, might be deduced from Weyl's asymptotic law adapted to the present setting, cf. for instance Milman's formula (2.9) in \cite{milman}: as $\lambda \to +\infty$, the eigenvalue counting function behaves as
\begin{equation}
\label{eq:counting}
\# \{ k \geq 1 : \lambda_k (-L) \leq \lambda \} = C_\alpha \, \lambda ^{\frac{\alpha}{2(\alpha-1)}} \, (1+o(1) ) ,
\end{equation}
with $C_\alpha >0$ some explicit constant. We suspect however that the correct order of magnitude might be recovered by choosing more cleverly the $\varepsilon_i$.

\subsection{Lower bounding the eigenvalues through only one intertwining}
If we analyze in detail the proof of the lower bound of Theorem~\ref{theo:lambda-n}, we observe that at each step of the recursive argument we invoke Theorem~\ref{theo:johnsen} which is based on the intertwining of Lemma~\ref{lemme:entrelacement}. In other words, Theorem~\ref{theo:lambda-n} is obtained by proceeding to successive intertwinings of the type \eqref{eq:intert}. In this short part, we propose an alternative method, still based on this intertwining formula, but which uses it only once. More precisely, the idea is to use some convenient weight such that the operator $-L_{ab}$ in the right-hand-side of \eqref{eq:intert} has its spectrum which might be easier to analyze. Certainly, such an approach has no reason to entail optimality in general since $-L_{ab}$ is prescribed a priori but it has the advantage to simplify considerably the intertwining approach, in particular when $-L_{ab}$ is the Gaussian operator $-L^\rho$ for some convenient $\rho>0$, for which we have seen that $\lambda_n (-L^\rho)$ behaves linearly in $n$. Hence let us start from the original operator $-L$ and consider the weight $a_1= e^{-(V-Z)/2}$, where $Z$ is some smooth potential to be determined later. Then the intertwining of Lemma~\ref{lemme:entrelacement} applied with the couple $(a,b) = (1, a_1)$ gives
$$
L_{a_1} f = f'' - Z' \, f' , \quad \mbox{ and } \quad M_1 ^{a_1} = \frac{V''}{2} + \frac{{V'}^2}{4} + \frac{Z''}{2} - \frac{{Z'}^2}{4}.
$$
Hence if we know a priori that $\lambda_1 (-L) >0$ then Theorem~\ref{theo:johnsen} applies. Assuming that there exists $\kappa \in \R$ such that $\inf_{x\in \R} \,  M_1 ^{a_1} (x) \geq \kappa$, then we have for all $n\in \N^*$,
$$
\lambda_n (-L) \geq \lambda_{n-1} (- L_{a_1}) + \kappa .
$$
In particular if we choose a Gaussian potential $Z = \rho \vert \cdot \vert ^2 /2$ with $\rho >0$, then the following result holds.
\begin{theo}
\label{theo:one-intert}
Assume that $\lambda_1 (-L) >0$ and the existence of some constant $\kappa \in \R$ such that
\begin{equation}
\label{eq:crit_gaussien}
M_1 ^{a_1} = \frac{V''}{2} + \frac{{V'}^2}{4} + \frac{\rho}{2} - \frac{\rho ^2 x^2 }{4} \geq \kappa .
\end{equation}
Then we have for all $n\in \N^*$,
$$
\lambda_n (-L) \geq \lambda_{n-1} (-L ^\rho ) + \kappa = \rho (n-1) + \kappa .
$$
\end{theo}

Certainly, the lower bound above is meaningful only if the right-hand-side is positive. However an important point is to allow the constant $\kappa$ to be  non-positive, as we will see below in the examples. Note that if $\kappa >0$ then the assumption $\lambda_1 (-L) >0$ is redundant, according to Chen-Wang's formula \eqref{eq:lambda1} applied with $(a,b) = (1,a_1)$ (we suspect however that this redundance holds in full generality). \smallskip

In the case of a uniformly convex potential $V$, this result allows to recover the lower bound in \eqref{eq:compar_gauss} when comparing with the Gaussian setting. Assume that $V'' \geq \rho$ with $\rho >0$. At the price of a translation in space causing no trouble, let us assume without loss of generality that the unique minimum of $V$ is attained at the origin, so that we have
$$
\vert V'(x) \vert \geq \rho \, \vert x \vert , \quad x\in \R .
$$
Then the criterion \eqref{eq:crit_gaussien} indicates that $M_1 ^{a_1} \geq \rho$ and therefore we obtain for all $n\in \N^*$,
$$
\lambda_n (-L) \geq \lambda_{n-1} (- L^\rho) + \rho = \lambda_{n} (- L^\rho).
$$

Actually, the criterion \eqref{eq:crit_gaussien} is sufficiently robust to involve many various situations of interest that might be difficult to cover otherwise. Let us start by considering once again the Subbotin distribution through the strictly convex potential $V = \vert \cdot \vert ^\alpha /\alpha$ but this time with $\alpha >2$. As mentioned earlier, we know that $\lambda_1 (-L) >0$. We assume that $Z = \rho \vert \cdot \vert ^2 /2$ with $\rho >0$ to be determined later. Since at infinity we have $(V')^2 \gg V''$, we neglect the contribution of $V''$ in the inequality \eqref{eq:crit_gaussien}, so that we are looking for the minimum of the function
$$
x\mapsto \frac{\vert x\vert ^{2(\alpha-1)}}{4} - \frac{\rho^2 x^2}{4} + \frac{\rho}{2}, \quad x\in \R ,
$$
which is negative and equals $\kappa := -C_\alpha \, \rho^{\frac{2(\alpha-1)}{\alpha-2}} + \rho/2$, where $C_\alpha >0$ is some explicit constant only depending on $\alpha$. Therefore we obtain by Theorem~\ref{theo:one-intert} the following estimate: for all $n\in \N^*$,
$$
\lambda_n (-L) \geq \rho (n-1)-C_\alpha \, \rho^{\frac{2(\alpha-1)}{\alpha-2}} + \frac{\rho}{2}.
$$
Finally optimizing in $\rho >0$ yields for all $n \in \N^*$ the lower bound
$$
\lambda_n (-L) \geq \tilde{C}_\alpha \, n ^{\frac{2(\alpha-1)}{\alpha}},
$$
which is sharp as $n$ is large, according to Weyl's asymptotic law \eqref{eq:counting}. Above $\tilde{C}_\alpha  >0$ is some explicit constant that only depends on $\alpha$. \smallskip

Another interesting situation is the case of a somewhat degenerated potential of the type
$$
V(x) = \frac{x^2}{2} + \alpha \sin (\beta x^2), \quad x\in \R ,
$$
where $\alpha, \beta >0$ are some small parameters to be calibrated so that the criterion \eqref{eq:ess_empty} ensuring a discrete spectrum is satisfied. In other words, $V$ is some Gaussian potential perturbed by an oscillating (and bounded) function whose oscillation intensity depends on $\alpha$ and $\beta$. We have for all $x\in \R$,
$$
V'(x) = x \, \left( 1+ 2 \alpha \beta \cos (\beta x^2)\right) \quad \mbox{and} \quad V''(x) = 1 + 2 \alpha \beta \cos (\beta x^2) - 4 \alpha \beta ^2 x^2 \sin (\beta x^2),
$$
and therefore the degeneracy resides in the fact that $\liminf _{\vert x\vert \to +\infty } V''(x) = - \infty$. However provided at least $\alpha$ or $\beta$ is sufficiently small, there exists positive constants $\rho$ and $\kappa$ both depending on $\alpha$ or $\beta$ such that \eqref{eq:crit_gaussien} is satisfied. Hence Theorem~\ref{theo:one-intert} can be applied to this example.

\section{Estimating the gap between the two first positive eigenvalues beyond the uniformly convex case}
\label{sect:ecart}
As mentioned previously in Section \ref{sect:ex-ap}, Johnsen's theorem \ref{theo:Johnsen_initial} allows us to obtain in the uniformly convex case some estimates on the gap between consecutive eigenvalues. However the situation is much more complicated when dealing with non uniformly convex potentials. Letting $n\in \N^*$ be fixed, if we assume that $\lambda_i (-L) \in \sigma_{disc}(-L)$ for every $i \in \{ 1, \ldots, n-1 \}$, then we saw in Section \ref{sect:main} that \eqref{eq:lambda_an} (combined with \eqref{eq:gap_consecutive}) leads to the following identity:  with $a_0 = 1$,
$$
\lambda_n (-L) - \lambda_{n-1} (-L)  = \sup_{a_n \in \FF} \, \inf_{x\in \R} \, M_{a_0 \ldots a_{n-1}} ^{a_n} (x),
$$
the $a_i$ being given recursively by
$$
\partial_{a_i} \ldots \partial_{a_1} g_i = 1, \quad i\in \{ 1, \ldots, n-1 \},
$$
where each eigenfunction $g_i$ is associated to the eigenvalue $\lambda_i (-L)$, $i\in \{ 1, \ldots, n-1 \}$. Hence the gap $\lambda_n (-L) - \lambda_{n-1} (-L)$ in the spectrum of the operator $-L$ can be estimated as soon as we are able to control the potential $M_{a_0 \ldots a_{n-1}} ^{a_n}$ for some well-chosen weight $a_n \in \FF$. As expected, this problem reveals to be more difficult than only estimating the eigenvalues $\lambda_{n} (-L)$ themselves as proposed in Section~\ref{sect:main}. Indeed in this case the weights $a_i$ are unknown since the eigenfunctions $g_i$ of the operator $-L$ are unknown in general. In this final part of the paper, we introduce a strategy to cover the case $n=2$, i.e. to estimate from below the gap between the two first positive eigenvalues. \smallskip

Letting $n=2$, we assume that $\lambda_1 (-L) \in \sigma_{disc} (-L)$ and thus we focus our attention on the control of a potential of the type $M_{a_1} ^{a_2}$ for some weight $a_2\in \FF$ (the artificial weight $a_0$ has been removed to lighten the notation) and with $a_1$ given by
$$
\partial_{a_1} g_1 = 1, \quad i.e., \quad a_1 = \frac{1}{g_1 '} .
$$
Hence the question is to find such a relevant function $a_2 \in \FF$ such that we have $\inf_{x\in \R} \, M_{a_1} ^{a_2} (x) >0$ from which would follow the desired estimate
$$
\lambda_2 (-L) - \lambda_{1} (-L) \geq \inf_{x\in \R} \, M_{a_1} ^{a_2} (x).
$$
For instance the choice $a_2 =1$ entails that $M_{a_1} ^{a_2} = V_{a_1} ''$ and therefore the requirement $\inf_{x\in \R} \, M_{a_1} ^{a_2} (x) >0$ reduces to the Bakry-\'Emery criterion, that is, the uniform convexity of the potential $V_{a_1}$, which has no explicit expression since $g_1$ is unknown in general. To overcome this difficulty, the idea is to use Theorem~\ref{theo:johnsen} with some convenient weight so that the devoted Schr\"odinger operator is nothing but the Laplacian perturbed by a certain potential, and then to use a celebrated result of Brascamp and Lieb about the log-concavity of the associated ground state. More precisely when choosing the couple $(a,b) = (1, e^{-V/2})$, we have by Theorem~\ref{theo:johnsen} for the first positive eigenvalues,
$$
\lambda_1 (-L) = \lambda_0 (- \Delta + M_1 ^b) .
$$
Above $\Delta$ stands for the Laplacian $\Delta f = f''$ and the potential $M_1 ^b$ is
\begin{eqnarray*}
M_1 ^b & = & V'' - e^{-V/2} \, L (e^{V/2}) \\
& = & V'' + e^{V/2} \, \Delta (e^{-V/2}) \\
& = & \frac{V''}{2} + \frac{(V')^2}{4}.
\end{eqnarray*}
Since $\lambda_1 (-L) \in \sigma_{disc} (-L)$ we have $\lambda_0 (- \Delta + M_1 ^b) \in \sigma_{disc} (- \Delta + M_1 ^b)$ and therefore the ground state $\tilde{g}_0 ^b$ exists and is of constant sign and non-vanishing, cf. Lemma~\ref{lemme:g1}. Moreover, as already mentioned in the proof of Lemma~\ref{lemme:g1}, the intertwining of Lemma~\ref{lemme:entrelacement} allows us to write
$$
(- \Delta + M_1 ^b) \, (\partial_b g_1) = - \partial_b L g_1 = \lambda_1 (-L) \, \partial_b g_1 = \lambda_0 (- \Delta + M_1 ^b) \, \partial_b g_1 ,
$$
so that the eigenfunctions $\tilde{g}_0 ^b$ and $g_1$ are connected as follows: we have $\tilde{g}_0 ^b = \partial_b g_1$ (up to some change of sign). Then the famous result of Brascamp and Lieb \cite{brascamp_lieb} adapted to the present framework (in particular they consider $\Delta /2$ instead of $\Delta$, requiring in our framework a slight modification of the constants they obtain) states that the ground state $\tilde{g}_0 ^b$ is uniformly log-concave as soon as $M_1 ^b$ is uniformly convex, that is, if $M_1 ^b$ satisfies $(M_1 ^b)'' \geq 2 \kappa ^2$ for some $\kappa >0$, then the ground state satisfies $(- \log \, \tilde{g}_0 ^b)'' \geq \kappa$. However we also have
$$
- \log \tilde{g}_0 ^b = - \log \partial_b g_1 = - \log \left( e^{-V/2} g_1 '\right) = \frac{V}{2} - \log \left( g_1 '\right) = \frac{V_{a_1}}{2},
$$
thus leading to $V_{a_1} '' \geq 2 \kappa$, i.e., the measure $\mu_{a_1}$ is uniformly log-concave. Finally combining all the preceding computations, we get the following result.
\begin{theo}
\label{theo:ecart}
Assume that $\lambda_1 (-L) \in \sigma _{disc} (-L)$. If the potential $V'' / 2 + (V')^2 / 4$ has it second derivative bounded from below by $2 \kappa ^2$ for some $\kappa >0$, then the following estimate between the two first positive eigenvalues holds:
$$
\lambda_2 (-L) - \lambda_1 (-L) \geq 2\kappa .
$$
\end{theo}
Beyond the standard Gaussian case for which Theorem~\ref{theo:ecart} is sharp, the assumptions allow us to consider non uniformly convex potentials and even non-convex situations. For instance with $V = \vert \cdot \vert ^4 /4$ (recall that in this case we have $\lambda_1 (-L) \in \sigma _{disc} (-L)$, the spectrum being discrete according to \eqref{eq:ess_empty}), we have
$$
\frac{V''}{2} + \frac{(V')^2}{4} = \frac{3 x^2}{2} + \frac{x^6}{4} ,
$$
and we get $\kappa = \sqrt{3/2}$ so that
$$
\lambda_2 (-L) - \lambda_1 (-L) \geq \sqrt{6}.
$$
For the double-well potential $V = \vert \cdot \vert ^4 /4 - \beta \vert \cdot \vert ^2 /2$, exhibiting a concave region in a neighbourhood
of the origin controlled by the size of the parameter $\beta >0$ (once again \eqref{eq:ess_empty} holds for this example and thus $\lambda_1 (-L) \in \sigma _{disc} (-L)$), we have
$$
\frac{V''}{2} + \frac{(V')^2}{4} = \frac{(\beta^2 + 6) x^2}{4} + \frac{(1- 2\beta) x^4}{4} - \frac{\beta}{2},
$$
and provided $\beta \in (0,1/2] $ we get $\kappa = \sqrt{\beta^2 + 6} \, /2$. Finally we obtain
$$
\lambda_2 (-L) - \lambda_1 (-L) \geq \sqrt{6+\beta^2}.
$$

As a concluding remark of this work, we mention that the procedure above might be iterated for all $n\in \N^*$, the main assumption being the uniform convexity of the potential $V_{a_1 \ldots a_{n-1}} '' / 2 + (V' _{a_1 \ldots a_{n-1}})^2 / 4$, which is not satisfactory for the moment since the $a_i$, which are defined through the eigenfunctions $g_i$, are unknown in general. However we believe that the current approach through the intertwinings is promising for the future and will allow to address this problem in full generality, for instance by choosing more conveniently the weight $a_2$.

\end{document}